\documentclass[reqno,11pt]{amsart}
\usepackage{amssymb}
\usepackage{amsmath}
\usepackage{amsthm}
\usepackage{color}
\usepackage{pifont,latexsym,ifthen,amsthm,rotating,calc,textcase,booktabs,cancel,slashed}

\usepackage{appendix}
\makeatletter
\def\@currentlabel{2.1}\label{e:dispaa}
\def\@currentlabel{2.21}\label{e:dispau}
\def\@currentlabel{2.22}\label{e:dispav}
\def\@currentlabel{2.23}\label{e:dispaw}
\def\@currentlabel{2.24}\label{e:dispax}
\def\theequation{\thesection.\@arabic\c@equation}
\makeatother

\newcounter{mnotecount}[section]

\newcommand{\rmnote}[1]{}

\renewcommand{\theequation}{\arabic{section}.\arabic{equation}}
\newtheorem{thm}{Theorem}[section]

\theoremstyle{definition}

\newtheorem{theorem}{Theorem}[section]
\newtheorem{lemma}[theorem]{Lemma}
\newtheorem{corollary}[theorem]{Corollary}
\newtheorem{proposition}[theorem]{Proposition}

\newtheorem{remark}[theorem]{Remark}

\newcommand{\Rm}{{\mathbb R}}

\begin{document}
\title[Scattering for subcritical wave equation with potential]{Scattering theory for  the subcritical wave equation with inverse square potential
}

\author[Miao]{Changxing Miao}
\address{Changxing Miao:
\hfill\newline Institute of Applied Physics and Computational
Mathematics, \hfill\newline P. O. Box 8009,\ Beijing,\ China,\
100088,} \email{miao\_changxing@iapcm.ac.cn}

\author[Shen]{Ruipeng Shen}
\address{Ruipeng Shen:
 \hfill\newline  Centre for Applied Mathematics,
Tianjin University,
Tianjin, China
} \email{srpgow@163.com}

\author[Zhao]{Tengfei Zhao}
\address{Tengfei Zhao:
 \hfill\newline School of Mathematics and Physics,
University of Science and Technology Beijing,
Beijing 100083, China} \email{zhao\_tengfei@ustb.edu.cn}


\maketitle

\begin{abstract}
We consider the scattering theory for the defocusing energy subcritical wave equations with an inverse square potential.
By employing  the energy flux method 
we establish energy  flux estimates on the light cone.
Then by the characteristic line method and radiation theorem, we show that the radial  finite-energy  solutions scatter to free waves outside of light cones.
Using Morawetz estimates we then obtain
the scattering theory for radial solutions with finite weighted energy initial data.
\vskip0.15cm

\end{abstract}

\section{Introduction}
\setcounter{section}{1}\setcounter{equation}{0}

We consider the nonlinear wave equation with an inverse-square potential in $\Rm^d$,
i.e.,
\begin{equation}
 \label{wave-La-p}
 \left\{\begin{aligned}&  u_{tt} +\Big(- \Delta + \frac{a}{|x|^2} \Big) u = - |u|^{p-1} u, \qquad (x,t)\in \Rm^d \times \Rm,\\
& (u,u_t)|_{t=0} = (u_0,u_1) \in \dot{H}^1 \times L^2,\end{aligned}\right.
\end{equation}
where $d\geq 3$, $p \in [1+\frac{4}{d-1},1+\frac{4}{d-2})$ and $a >-\frac{(d-2)^2}{4} $.
We define
$$\mathcal L_a= -\Delta +\frac{a}{|x|^2}$$
as the Friedrichs extension of the quadratic form defined on $C_c^\infty(\mathbb R^d\backslash\{0\})$ via
\[
g\mapsto \int_{\mathbb R^d} \big(|\nabla g(x)|^2 + a|x|^{-2}|g(x)|^2 \big)\,dx.
\]
 The restriction $a>-\frac{(d-2)^2}{4}$ ensures that the operator $\mathcal L_a$ is positive
definite in $\dot H^1_a(\mathbb R^d)$  via  the sharp Hardy inequalities, and one  finds that the standard Sobolev space $\dot H^1$ is equivalent to the Sobolev space $\dot H_a^1$ defined in terms of $\mathcal L_a$ (see Section~\ref{basic}).  In particular,  $\mathcal L_a=-\Delta$ when $a=0$.

The operator $\mathcal L_a$ arises often in mathematics and physics in scaling limits of more complicated problems in combustion theory, the Dirac equation with Coulomb potential, and the study of perturbations of space-time metrics such as Schwarzschild and Reissner-Nordstr\"om \cite{BPST-2003, KSWW, VazZua, ZhaZhe}.
One interesting feature of the inverse-square potential is that it shares the same scaling with the Laplacian.  In general, one cannot  treat $\mathcal L_a$ as a
 perturbation of $-\Delta$, adding to the mathematical interest of this particular model.

It is easy to find that the solution $u$ to the equation \eqref{wave-La-p} 
satisfies an energy conservation law:
\begin{align*}
  E(u,u_t)&\doteq\int_{\Rm^d} \Big( \frac{1}{2} |\nabla u(x,t) |^2+ \frac{a}{2} \frac{|u(x)|^2}{|x|^2}+ \frac{1}{2} |\partial_t u(x,t)|^2 +\frac{1}{p+1} |u(x,t)|^{p+1}\Big)dx\nonumber\\
  & =E(u_0,u_1).
\end{align*}
If $u$ is a solution to \eqref{wave-La-p}, then, for $\lambda>0$,
 the associated scaling transform
  $u_\lambda(x,t)=\lambda^{-\frac{2}{p-1}} u( \tfrac{x}{\lambda}  ,\tfrac{t}{\lambda} )$
 is also a solution and its  norm
  in space $\dot H^{s_p}(\Rm^d)$
  is invariant, where $s_p= \tfrac{d}{2}-\tfrac{2}{p-1}$.
When $s_p=1$(i.e., $p=p_e\doteq1+\frac{4}{d-2}$), we say that the equation \eqref{wave-La-p} is energy critical. We say that the equation \eqref{wave-La-p} is energy subcritical(or supercritical)  for $s_p<1$(or $s_p>1$). In particular, when $p$ equals the conformal exponent  $p_{conf}\doteq1+\frac{4}{d-1}$, we have
$s_p=\frac12$.

 \subsection{Background and Motivations}

 The presence of the inverse-square potential breaks translation symmetry, introducing new challenges into the analysis of \eqref{wave-La-p}.
 One has to work  in  the context of  dispersive equations in the presence of broken symmetries.
 We first recall the study of  the defocusing nonlinear wave equation
\begin{equation}
 \label{wave-p}
 \left\{\begin{array}{l}  u_{tt} - \Delta u = - |u|^{p-1} u, \qquad (x,t)\in \Rm^d \times \Rm,\\
 (u,u_t)|_{t=0} = (u_0,u_1) \in \dot{H}^1 \times L^2,\end{array}\right. 
\end{equation}
where $d\geq 3$ and $p \in [1+\frac{4}{d-1},1+\frac{4}{d-2}]$.
For the scattering theory of the solution $u$,
we mean that there exists a free wave solution $u_\pm$, i.e., a solution to the linear wave equation $u_{tt} - \Delta u = 0$,  such that
$u-u_\pm $ tend to zero in some suitable space as $t\to\pm\infty$.
For
 the
 energy critical case $p=p_e$, the
global well-posedness and scattering
theory of finite energy solutions to the equation \eqref{wave-p} has been extensively studied by Grillakis
\cite{Grillakis-1990, Grillakis-1992},
Shatah-Struwe \cite{SS-1993,SS-1994},
Bahouri-Shatah \cite{BS-1998}, and
Bahouri-G\'{e}rard \cite{BG-1999}.
 In the case of $p\neq p_e$, there is no conservation law  of same scaling of $\dot H^{s_p}\times \dot H^{s_p-1} $ that provides
 an a priori bound on the critical norm of the solution.
 Under the assumption that 
 the $\dot H^{s_p} \times \dot H^{s_p-1} $ norm of the solution is uniformly bounded, many authors
  studied  the scattering problem of solution for the problem \eqref{wave-p},
   see for example
 \cite{KM-2011,Shen-2014,DL-2015-1,DL-2015-2,DLMM-2018,Rodriguez-2017}.  We also refer to
  \cite{DR-2017, DY-2018}
   for general norms.

 For the defocusing subcritical case,  one may obtain the global well-posedness  for the Cauchy problem \eqref{wave-p} in the energy space
 via the Strichartz estimates and energy conservation law,  see for example \cite{LS-1985}.
 Nevertheless, in this case the scattering theory in  energy spaces remains open.
 Ginibre-Velo \cite{GV-1987} and   Hidano \cite{Hidano-2003} proved that if the initial datum are in the conformal energy spaces,
 then the solutions scatter in both time directions. In \cite{Shen-2017}, Shen
 established
  the scattering results for solutions in larger weighted spaces.
Later, Dodson
obtained
 the scattering results in
the critical Sobolev spaces for
solutions with radial initial datum in the critical space
$\dot B^2_{1,1}\times\dot B^1_{1,1}$ for $(d,p)=(3,3)$
 in \cite{Dodson-2018-APDE}
and later  $\dot H^{s_p} \times \dot H^{s_p-1} $,
 for $d=3$ and $3\leq p<5$ in \cite{Dodson-2018-1,Dodson-2018-2}.
Miao-Yang-Zhao \cite{MYZ-2019}  considered  the case $(d,p)=(5,2)$ for initial datum in the critical space $\dot B^3_{1,1}\times\dot B^2_{1,1}.$
Recently,  Shen developed the energy flux method and obtained scattering results for solutions  with
 initial datum in the
weighted energy spaces in \cite{Shen-2018-1,shenenergy,Shen-2019-1,shenhdradial}.
For the non-radial scattering results, we refer to
 Shen \cite{shen3dnonradial,shenhd} and Yang \cite{Yang-2019}.

\vskip0.3cm

 Now we consider the Cauchy problem of the defocusing wave equation with inverse square potential \eqref{wave-La-p},
 which has  attracted a great deal of interest in recent years, see, e.g., \cite{BPST-2003, KMVZZ-2018, MMZ-2019}
 and in particular \cite{KMVZZ-2018, KMVZZ-NLS, Mur,ZhaZhe} for the case of nonlinear Schr\"odinger equations with an inverse-square potential.
 Burq, Planchon, Stalker and Tahvildar-Zadeh \cite{BPST-2003} proved the  well-posedness of  \eqref{wave-La-p} by establishing the Strichartz estimates
associated with $\mathcal L_a$.  We also refer to Killip-Miao-Visan-Zheng-Zhang \cite{KMVZZ-2018} for the development of
 harmonic analysis tools of the operator $\mathcal{L}_a$ by utilizing the heat kernel bounds.
Later, by employing these tools and the concentration compactness arguments, Miao-Murphy-Zheng \cite{MMZ-2019} proved
 the scattering of solutions to the problem \eqref{wave-La-p}
in the critical case,  i.e., the solution  approaches  the linear solutions associated with the operator $\vec S_a(t) $.
Here $\vec S_a(t) $ denotes the solution operator of the linear equation
$\partial_t^2 u+\big(-\Delta +\tfrac{a}{|x|^2}\big)u=0$. However, the scattering of the energy subcritical situation seems still open for us.

\subsection{Main Results}
This paper is devoted to  the study of scattering theory of solutions for  \eqref{wave-La-p} in sub-critical cases.
We prove the following scattering result for solutions of   \eqref{wave-La-p}  with some additional energy decay.

\begin{theorem} \label{main 2}
Suppose that $3\leq d \leq 6$, $p\in[1+\frac{4}{d-1},1+\frac{4}{d-2})$,
and $a>-\frac{(d-2)^2}4+\big(\frac{(d-2)p-d}{2p}\big)^2$.
Let $u$ be a radial solution to \eqref{wave-La-p} with initial data $(u_0,u_1) \in \dot{H}^1 \times L^2$
so that the inequality
\begin{equation}\label{weight-energy}
 E_{\kappa} (u_0,u_1) \doteq \int_{\Rm^d} (|x|^\kappa + 1)\left(|\nabla u_0(x) |^2 +|u_1(x)|^2 + |u_0(x)|^{p+1} \right) dx < +\infty
\end{equation}
holds for a constant $\kappa \geq \kappa_0 \doteq \frac{(d+2)-(d-2)p}{p+1}$. Then the solution $u$ scatters, that is,  there exists a finite-energy free wave $v^\pm$,  so that
\[
 \lim_{t \rightarrow \pm \infty} \int_{\Rm^d} \left(|\nabla u(x,t) - \nabla v^\pm(x,t)|^2 + |u_t(x,t)-v^\pm_t(x,t)|^2 \right) dx = 0.
\]
\end{theorem}

\begin{remark}\label{rem-1.3}
We remark that  Theorem \ref{main 2} depend on the  local theory with $(u,u_t) \in C(\Rm; \dot{H}^1 \times L^2)$ and $|u|^{p-1} u \in L_{\rm loc}^1 L^2(\Rm \times \Rm^d)$.
 The condition $|u|^{p-1} u \in L_{\rm loc}^1 L^2 (\Rm \times \Rm^d)$ will be very useful when we need to apply a smooth approximation argument.
The most convenient  Strichartz space is  $L_t^{\frac{2p}{(d-2)p-d} } L^{2p}_x $.  When  $d=3$,
$\frac{2p}{(d-2)p-d}=\frac{2p}{p-3}>2$,  which holds for $p\in[3,5)$.
When  $d\geq 4$, the conditions
$$\tfrac{2p}{(d-2)p-d}\geq 2, \;\;\text{\rm and}\;\;p\in[p_{conf},p_e) $$
imply that we need the restriction $d\leq 6.$ For the higher dimension $d\geq7$, one may consider the exotic Strichartz estimates as in \cite{Foschi-2005, Visan-2007}.
On the other hand, the assumption of $a$ in Theorem \ref{main 2} originates from  the Strichartz estimates in Proposition \ref {Strichartz}.
In fact,
$\big(\tfrac{(d-2)p-d}{2p}\big)^2$ is increasing in $p$ and
corresponds to
the exponents in
\cite{MMZ-2019} when $d=3,4$ and $p=p_e$.
\end{remark}




Now, 
 we   outline  the proof of Theorem \ref{main 2}  and do a few suitable reductions.
For the case of  $a>0$,
 the proof  is similar to the case $a=0$ of \cite{shenhdradial}. In fact, one can show the energy flux estimates  on the light cones  by using the energy flux formula
\begin{align}\nonumber
  &E(t_2; B(0,t_2-\eta)) - E(t_1; B(0,t_1-\eta)) \\
  =& \frac{1}{\sqrt{2}} \int_{|x|=t-\eta, t_1<t<t_2} \Big(\frac{1}{2}|u_r+u_t|^2 +\frac{a}{2} \frac{|u|^2}{|x|^2} + \frac{1}{p+1}|u|^{p+1}\Big) dS.
\label{energy-flux-a>0}
 \end{align}
By making use of 
  the characteristic line method, one can show Theorem \ref{main 1} holds since in this case 
   $a \frac{u}{|x|^2}$
   is a defocusing term with a good spatial decay.
Next, for the solution with weighted energy, one can show the
scattering results by utilizing the energy distribution estimates
\begin{align}\nonumber
 & \int_{R<|t|<T} \int_{|x|< R} \left(\frac{1}{2}|\nabla u|^2 +\frac{a}{2} \frac{|u|^2}{|x|^2}+\frac{1}{2}| u_t|^2 + \frac{1}{p+1}|u|^{p+1}\right) dxdt \\
 \leq& \int_{-R}^R \int_{|x|> R} \left(\frac{1}{2}|\nabla u|^2 +\frac{a}{2} \frac{|u|^2}{|x|^2}+\frac{1}{2}| u_t|^2 + \frac{1}{p+1}|u|^{p+1}\right) dxdt
 ,
 \label{coro-Morawetz-a>0}
 \end{align}
  which follow from the Morawetz estimates
\begin{align}\nonumber
  &\frac1{2R} \int_{T_1}^{T_2} \int_{|x|< R}
  \left(|\nabla u|^2 +a \frac{|u|^2}{|x|^2}+| u_t|^2  + \frac{(d-1)(p-1)-2}{p+1} |u|^{p+1} \right) dxdt  \\
 & + \frac{d-1}{4R^2}\int_{T_1}^{T_2} \int_{|x|=R} |u|^2dS(x)dt  +   \int_{T_1}^{T_2}\int_{|x|\geq R} \left(  \frac{(d-1)(p-1)}{2(p+1)}  \frac{|u|^{p+1}}{|x|}+  (a+\lambda_d) \frac{|u|^2}{|x|^3}\right) dxdt\nonumber \\
 \leq& 2E ,\label{Morawetz-a>0}
  \end{align}
for some constant $\lambda_d\geq0.$ For more details, we refer to \cite{shenhdradial}.

We only focus on the case $a<0$ from now on. From  the  formula
\eqref{energy-flux-a>0}, we prove the energy flux estimates on light cones by using the Hardy
 inequality and the radial decay properties in Lemma \ref{pointwise estimate 2} of the radial solutions.
Thanks to these estimates and the radiation 
field for free waves, we are able to prove the following proposition by utilizing the characteristics lines method,
which corresponds to the scattering result
in the region outside the light cone.

\begin{proposition}[Exterior scattering] \label{main 1}
Suppose that $3\leq d \leq 6$, $p\in[1+\frac{4}{d-1},1+\frac{4}{d-2})$,
and $a>-\frac{(d-2)^2}4+\big(\frac{(d-2)p-d}{2p}\big)^2$.
Let $u$ be a radial solution to \eqref{wave-La-p} with a finite energy. Then there exist finite-energy free waves $  v^{\pm}$(solutions to the linear wave equation $\partial_t^2 v - \Delta v=0$)
so that
\[
 \lim_{t \rightarrow \pm\infty} \int_{  |x|>|t|-\eta} \left(|\nabla u(x,t) - \nabla v^{\pm}(x,t)|^2 + |u_t(x,t)- v^{\pm}_t(x,t)|^2 \right) dx = 0, \;\forall \eta \in \Rm.
\]
\end{proposition}

Next,
for the proof of Theorem \ref{main 2}, we need to establish  the decay estimate of the energy in the interior of light cones.   Since the
inequality \eqref{Morawetz-a>0}   seems not  to hold for $a<0$,  we establish a  modified version of
Morawetz estimates, which heavily depends on the Hardy inequalities on a local region.
Employing the aforementioned modified Morawetz estimates, one can prove the energy decay
in the interior of light cones
by introducing a cut-off version of the kinetic energy $J^{\pm}_u$.
By making use of the improved pointwise estimate \eqref{basic-eq-4} in Lemma \ref{pointwise estimate 2} and
the energy decay of $u$ in exterior of light cone,  we prove the
scattering results for solutions with an additional energy decay.


\begin{remark}
  As noted in \cite{shenhdradial}, there exists  radial smooth initial datum $(u_0,u_1)$ such that the assumptions of Theorem \ref{main 2} hold, but $(u_0,u_1)$ do not belong to the scaling invariant spaces  $\dot H^{s_p}\times \dot H^{s_p}(\Rm^d) $.
  In fact, suppose
 \begin{equation}
  u_0(x)=(1+|x|)^{-\frac{2(p+d+1)}{(p+1)^2}-\varepsilon} \text{ ~~and ~~} \left|\nabla u_0(x)\right| \approx (1+|x|)^{-\frac{2(p+d+1)}{(p+1)^2}-1-\varepsilon},   |x|\gg1.
 \end{equation}
 By a direct computation, we have $(u_0,0)$ satisfies the assumption of
 Theorem \ref{main 2} with $\kappa=\kappa_0$. But $u_0 \not\in L^\frac{d(p-1)}{2} (\Rm^d)$ when $\varepsilon$ is sufficiently small.  Thus $u_0 \not\in \dot H^{s_p}(\Rm^d)$ by the Sobolev embedding.

\end{remark}

\begin{remark}
For the case $-\lambda_d<a$, one may improve the
 result of Theorem \ref{main 2} slightly by showing the inward/outward energy flux and
the weighted Morawetz estimates as in  \cite{shen3dnonradial}. Here $\lambda_d = (d-1)(d-3)/4$.
But the minimal decay rate $\kappa_0$ of
energy in \eqref{weight-energy}  can not be further improved by the inward/outward energy theory.
While for the case $a<-\lambda_d$, the inward/outward energy flux formula may be not true
since the extra term
$\iint_{\Rm\times\Rm^3} \frac{|u|^2}{|x|^3}dxdt$ may be infinite.
\end{remark}


\begin{remark}

In the appendix we consider the scattering
theory (which may be of interest in its own right) for the linear wave equation with inverse
square potential
  \[ \label{wave-Linear}
 \left\{\begin{array}{l}  u_{tt} +\left(- \Delta + \frac{a}{|x|^2} \right) u = 0, \qquad (x,t)\in \Rm^d \times \Rm,\\
 (u,u_t)|_{t=0} = (u_0,u_1),\end{array}\right. \tag{$LW_a$}
\]
where $d\geq3$ and $a>-\frac{(d-2)^2}{4}$.
By the Strichartz estimates of $\vec S_0(t)$ and the generalized Morawetz estimates, one may obtain the scattering result of solution of \eqref{wave-Linear} in $\dot H^\frac12 \times \dot H^{-\frac12}(\Rm^d)$. 
In fact, by the Duhamel formula, for $t_1<t_2 $, we have
\begin{equation}\label{Duhamel-1}
 \vec S_0(-t_2)(u(t_2),u_t(t_2))(x) - \vec S_0 (-t_1)(u(t_1),u_t(t_1))(x)  =
\int_{t_1}^{t_2} \vec S_0(-s) \big[0, \tfrac{a}{|\cdot|^2} u(\cdot,s) \big](x)ds.
\end{equation}
Then, by the square function estimates\footnote{ One can also use the Strichartz estimates on $L^2_t L^{\tfrac{2d}{d-2},2 }_x $ here but with radial assumption if $d=3$. For the cases $d\geq 4,$ we refer to \cite{KT-1998}. If $d=3$, 
 this follows from the 
 Strichartz estimates for radial data
(\cite[Theorem 1.3]{Sterbenz2005}) and real interpolation estimates in \cite[Section 6]{KT-1998}.}
presented in the appendix and the generalized H\"older inequality, we have
\begin{align}
      &  \left\| \vec S_0(-t_1)(u(t_1),u_t(t_1)) - \vec S_0(-t_2)(u(t_2),u_t(t_2))  \right\|_{\dot H^\frac12  \times \dot H^{-\frac12} (\Rm^d)}\nonumber\\
  \lesssim~ & \left\|\tfrac{a}{|x|^2}u(x,t) \right\|_{L^{\tfrac{2d}{d+2},2 }_x L^2_t ([t_1,t_2]  \times \Rm^d ) }\nonumber\\
  \lesssim ~& \left\| |x|^{- 1}\right\|_{L^{d ,\infty}_x(\Rm^d)} \left\| \tfrac{a}{|x| } u(x,t) \right\|_{L^2_{t,x} (\Rm^d\times[t_1,t_2])} \to 0,
  \label{Duhamel-3}
\end{align}
as $t_2>t_1\to \infty,$
where $L^{\alpha,\beta}_x$ are Lorentz spaces and we have used   the generalized Morawetz estimate  as in \cite[Theorem 2]{BPST-2003}, i.e.,
\begin{equation}
\|  \tfrac{1}{|x| } u(x,t) \|_{L^2_{t,x} (\Rm\times\Rm^d)}  \lesssim \|(u_0,u_1)\|_{\dot H^\frac12  \times \dot H^{-\frac12} (\Rm^d)}.
\end{equation}
However, 
 generally speaking, 
for scattering results in the energy space $\dot H^1\times L^2(\Rm^d)$,
 one may need the associated square function estimates on $L^{\frac{2d}{d-1},2}_xL^2_t(\Rm\times\Rm^d)$ or
the Strichartz estimates on $L^2_tL^{\frac{2d}{d-1},2}_x(\Rm\times\Rm^d)$.
But,  both these tools are not true even for radial case. In fact, it is easy to see  that square function estimates do not hold via Plancherel and  the fact $| \widehat{ d\sigma}_{S^{d-1} }(\xi) | = O((1+|\xi|)^{\frac{1-d}{2}})$).  On  the other hand, by making use of  Proposition \ref{decay-out-cone} and H\"older inequalities,  we have
\begin{equation}
 \| |e^{it|\nabla|} f \|_{L^q_x{ (||x|-|t||\lesssim 1) }}  \gtrsim
  \| |e^{it|\nabla|} f \|_{L^2_x{ (||x|-|t||\lesssim 1) }} |t|^{-(d-1)(\frac12-\frac1q)}
  \gtrsim |t|^{-(d-1)(\frac12-\frac1q)}.
 \end{equation}
 This inequality 
  shows 
  us that  $L^p_tL^q_x(\Rm\times\Rm^d)$  is acceptable if
 $\frac{1}{p}+\frac{d-1}{q}<\frac{d-1}{2}.$

 Fortunately, we can obtain scattering theory of radial finite energy solutions to the linear equation \eqref{wave-Linear} in the appendix by the argument of \cite{shenenergy} and a series of estimates. This fact shows us that  the radial solutions of the energy critical nonlinear wave equations scatter to free waves, see \cite{MMZ-2019}.

%

\end{remark}



\subsection{The structure of this paper}
 This article is organized as follows: In Section \ref{basic}, we give
technical lemmas  and well-posedness theory of Cauchy problem  \eqref{wave-La-p}.
 In Section \ref{energyflux}, we prove the energy flux formula of the solution to the Cauchy problem \eqref{wave-La-p}.
 We establish a class of local Morawetz estimates associated with the Cauchy problem\eqref{wave-La-p} in Section \ref{Morawetz-sec}.
In Section \ref{mtheod-characteristic}, we prove  Proposition \ref{main 1} by making use of the characteristic lines method developed in \cite{shenenergy}.
In Section \ref{further}  we establish a series of local energy estimates,
 and complete the proof of Theorem \ref{main 2} by making use of  Morawetz estimates established in Section  \ref{Morawetz-sec}. In the appendix, we prove the asymptotic behaviour of linear solution,
  and establish the square function estimate in Lorentz space by Stein-Tomas restriction theorem.

We conclude this section by introducing some notations. We denote by $A\lesssim B$ the fact that the inequality $A\le C B$ holds with an implicit constant C depending on $p$ and $a$. We denote by  $A \approx B$ if $A\lesssim B$ and $B \lesssim A$.
Sometimes we also allow these constants to depend on the energy $E$  and weighted energy $E_\kappa$, especially in the last two sections.

The derivative operator $\nabla$ refers to the spatial  variable only. $\partial_r=\frac{x}{|x|} \cdot \nabla$ denotes the derivative in the radial direction  and $\slashed{\nabla}=\nabla - \frac{x}{|x|} \partial_r$ denotes the angular derivative.
Let $u(x, t)$ be a spatially radial function. By convention $u(r, t)$ represents
the value of $u(x, t)$ when $|x|=r$.


\section{Technical Lemmas and well-posedness theory}\label{basic}
\setcounter{section}{2}\setcounter{equation}{0}

First, we recall the sharp Hardy inequality, for $d\geq 3,$
\begin{equation}
  \frac{(d-2)^2}{4} \int_{\Rm^d} \frac{|u(x)|^2}{|x|^2} dx\leq
  \int_{\Rm^d} |\nabla u(x)|^2 dx.\label{basic-eq-1}
\end{equation}

 For $a \in(-\frac{(d-2)^2}{4},\infty  )$,  we have the rough identity as follows
\begin{equation} E \approx \int_{\Rm^d} \Big(|\nabla u(x,t)|^2 + \frac{|u(x,t)|^2}{|x|^2}+| u_t(x,t)|^2 + |u(x,t)|^{p+1}\Big) dx. \label{basic-eq-2}
\end{equation}
Thus for any region $\Sigma \subset \Rm^d$ we have
\begin{equation}
 \int_{\Sigma} \Big(\frac{1}{2}|\nabla u(x,t)|^2 +\frac{a}{2} \frac{|u(x,t)|^2}{|x|^2}+\frac{1}{2}| u_t(x,t)|^2 + \frac{1}{p+1}|u(x,t)|^{p+1}\Big) dx \lesssim E.
\label{basic-eq-3}
\end{equation}
 \begin{lemma}[Pointwise Estimate, see Lemma 5.6 in \cite{shenenergy}] \label{pointwise estimate 2}
 All radial $\dot{H}^1(\Rm^d)$ functions $u$ satisfy
 $$  |u(r)| \lesssim_d  r^{-(d-2)/2} \|u\|_{\dot{H}^1}, \qquad r>0.$$
 If $u$ also satisfies $u \in L^{p+1} (\Rm^d)$, then its decay is stronger as $r \rightarrow +\infty$, i.e.,
 \begin{equation}
  |u(r)| \lesssim_d r^{-\frac{2(d-1)}{p+3}} \|u\|_{\dot{H}^1}^{\frac{2}{p+3}} \|u\|_{L^{p+1}}^{\frac{p+1}{p+3}}, \qquad r>0.
 \label{basic-eq-4}\end{equation}
\end{lemma}

Now, we recall some properties of the operator $\mathcal{L}_a=-\Delta +\frac{a }{|x|^2}.$
For $r\in(1,\infty)$, we write $\dot H^{1,r}_a$ and $H^{1,r}_a$ for the homogeneous and inhomogeneous space defined in terms of $\mathcal{L}_a$; these spaces have norms
$$\|f\|_{\dot H^{1,r}_a}= \|\sqrt{\mathcal{L}_a  } f \|_{L^r},\quad \;
\|f\|_{  H^{1,r}_a}= \|\sqrt{1+\mathcal{L}_a  } f \|_{L^r}. $$
We write $\dot H^{1,2}_a= \dot H^1_a$ and denote $\sigma= \frac{d-2}{2}-\sqrt{\left(\tfrac{d-2}{2} \right)^2+a } $.

\begin{lemma}[Equivalence of Sobolev spaces, see Theorem 1.2 in
\cite{KMVZZ-2018}]
 Let $d \geq 3$, $a > -\frac{(d-2)^2}{4}$, and $s\in(0,2)$.
\begin{itemize}
  \item \; If $p\in (1,\infty)$ satifies $\frac{s+\sigma}{d}<\frac1p<\min\left\{ 1,\frac{d-\sigma}{d}
      \right\}$, then
      \[  \||\nabla|^s f\|_{L^p(\mathbb{R}^d)} \lesssim \|\mathcal{L}^\frac{s}{2}_a  f\|_{L^p(\mathbb{R}^d)}.    \]
  \item \; If $p\in (1,\infty)$ satifies $\max\left\{  \frac{s}{d}, \frac{\sigma}{d} \right\}<\frac1p<\min\left\{ 1,\frac{d-\sigma}{d}
      \right\}$, then
      \[  \|\mathcal{L}^\frac{s}{2}_a  f\|_{L^p(\mathbb{R}^d)} \lesssim   \||\nabla|^s f\|_{L^p(\mathbb{R}^d)}.    \]
\end{itemize}
\end{lemma}

\subsection{Local theory}\label{local}
We consider the inhomogeneous wave equation
\begin{equation}\label{inh-wave-La}
 \left\{ \begin{aligned}
    \partial_t^2 u +\mathcal{L}_a u&=F(x,t),\quad (x,t)\in \mathbb{R}^d\times\mathbb{R} \\
    (u,u_t)(x,0)&=(u_0,u_1)\in \dot H^1\times L^2(\mathbb{R}^d).
  \end{aligned}\right.
\end{equation}
First, we record  the Strichartz estimates. 

\begin{proposition}[Strichartz estimates \cite{BPST-2003} \cite{MMZ-2019}]\label{Strichartz}
  Let   $ 2\leq q \leq \infty, 2\leq r < \infty$  such that the admissible condition
   $\frac{1}{q} +\frac{d-1}{2r} \leq \frac{d-1}{4}$ 
    holds,  
   where if $d=3$, we additionally require $q>2$.
  Define $\gamma$ via the scaling condition
      $  \frac1q+\frac{d}r   =\frac{d}2- \gamma.$
Then, we have
\begin{equation}\label{Strichartz-equ-1}
 \bigg\|\cos (t\sqrt{\mathcal{L}_a}) f + \frac{\sin (t\sqrt{\mathcal{L}_a)}}{\sqrt{\mathcal{L}_a}} g\bigg\|_{L^q_tL^r_x(\mathbb{R} \times \Rm^d)} \leq C \big(\|f\|_{\dot{H}^\gamma} + \|g\|_{\dot{H}^{\gamma-1}}\big)
\end{equation}
provided that
\[ -\min\left\{ 1,\sqrt{a +\tfrac94}-\tfrac12,\sqrt{a +\tfrac14}+1 \right\} <\gamma <\min\left\{
2,  \sqrt{a +\tfrac94}+\tfrac12,\sqrt{a+\tfrac14 }+1 -\tfrac1q \right\}  \] when  $d=3$;
\begin{align*}
 & -\min\left\{ \tfrac{d}{2} - \tfrac{d+3}{2(d-1)},\sqrt{a +\tfrac{d^2}4}-\tfrac{d+3}{2(d-1)},\sqrt{a +\tfrac{(d-2)^2}4}+1 \right\}\\
&  \qquad <\gamma  <\min\left\{
\tfrac{d+1}{2},  \sqrt{a +\tfrac{d^2}4}+\tfrac12,\sqrt{a+\tfrac{(d-2)^2}4 }+1 -\tfrac1q \right\}
\end{align*}  when $d\geq4$.
\end{proposition}

\begin{remark}
We will also need  
 an inhomogeneous  version of \eqref{Strichartz-equ-1},
\begin{equation}\label{inh-Strichartz}
 \Big\|\int_0^t \frac{\sin ((t-s)\sqrt{\mathcal{L}_a)}}{\sqrt{\mathcal{L}_a}} F(s) ds \Big\|_{L_t^\infty([0,T]; \dot{H}^1\times L^2) \cap L_t^q L_x^r ([0,T]\times \Rm^d)} \leq C \|F\|_{L^1_t L^2_x([0,T]\times \Rm^d) },
\end{equation}
where parameters $(q,r,\gamma) = (q,r,1)$ satisfy conditions in Proposition \ref{Strichartz}.
In fact,   by Strichartz estimates  we have
$$ \Big\|\chi_{t\geq s}(t) \frac{\sin ((t-s)\sqrt{\mathcal{L}_a)}}{\sqrt{\mathcal{L}_a}} F(s)\Big\|_{L_t^\infty([0,T]; \dot{H}^1\times L^2) \cap L_t^q L_x^r ([0,T]\times \Rm^d)} \leq C \|F(s) \|_{L^2_x(\Rm^d)}.
$$
This  inequality implies that
$$\Big\|\int_0^T \chi_{t\geq s}(t) \frac{\sin ((t-s)\sqrt{\mathcal{L}_a)}}{\sqrt{\mathcal{L}_a}} F(s) ds \Big\|_{L_t^\infty([0,T]; \dot{H}^1\times L^2) \cap L_t^q L_x^r ([0,T]\times \Rm^d)} \leq C \|F\|_{L^1_t L^2_x([0,T]\times \Rm^d) },$$
which is exactly inequality \eqref{inh-Strichartz}.
\end{remark}
For the well-posedness theory for 
the problem \eqref{wave-La-p}, we need to use the Strichartz norm such as
$\|\cdot\|_{L_t^{\frac{2p}{(d-2)p-d} } L^{2p}_x} $.
 Based on the discussion in Remark \ref{rem-1.3}, for each $p\in[p_{conf},p_e)$ with $3\le d\le6$,   we have by \eqref{inh-Strichartz}
\begin{align}\nonumber
  & \Big\| \int_{0}^t  \frac{ \sin((t-s)\sqrt{\mathcal{L}_a} )}{\sqrt{\mathcal{L}_a}} \big( |u|^{p-1} u \big)ds \Big\|_{L_t^\infty(\dot H^1_x\times L^2_x) \cap L_t^{\frac{2p}{(d-2)p-d} } L^{2p}_x (I\times \mathbb{R}^d) } \\
\leq &  C \| |u|^{p-1} u \|_{L^1_tL^2_x(I\times \mathbb{R}^d)} \nonumber\\
\leq &C  T^{\frac{(d+2)-(d-2)p}{2}} \|u\|^p_{ L_t^{\frac{2p}{(d-2)p-d} } L^{2p}_x(I\times \mathbb{R}^d)},\label{nonlinear-1}
\end{align}
and
\begin{align}\nonumber
  & \Big\| \int_{0}^t  \frac{ \sin((t-s)\sqrt{\mathcal{L}_a} )}{\sqrt{\mathcal{L}_a}} \Big( |u|^{p-1} u - |v|^{p-1} v \Big)ds \Big\|_{L_t^\infty(\dot H^1_x\times L^2_x) \cap L_t^{\frac{2p}{(d-2)p-d} } L^{2p}_x (I\times \mathbb{R}^d) } \\
\leq & ~~C  T^{\frac{(d+2)-(d-2)p}{2}} (\|u\|^{p-1}_{ L_t^{\frac{2p}{(d-2)p-d} } L^{2p}_x(I\times \mathbb{R}^d)}+\|v\|^{p-1}_{ L_t^{\frac{2p}{(d-2)p-d} } L^{2p}_x(I\times \mathbb{R}^d)}) \|u-v\|_{ L_t^{\frac{2p}{(d-2)p-d} } L^{2p}_x(I\times \mathbb{R}^d)},\label{nonlinear-2}
\end{align}
where $I=[0,T]$ and
$$a>-\tfrac{(d-2)^2}4+\left(\tfrac{(d-2)p-d}{2p}\right)^2.$$
Then, by the Duhamel formula
\[
u= \cos(t \sqrt{\mathcal{L}_a} ) u_0 +\sin(t\sqrt{\mathcal{L}_a} )u_1 -
\int_{0}^t  \frac{ \sin((t-s)\sqrt{\mathcal{L}_a} )}{\sqrt{\mathcal{L}_a}} \left( |u|^{p-1} u \right)ds,
\]
one can show the local well-posedness of \eqref{wave-La-p} via the concentration map argument.
Moreover, by energy conservation, one can obtain the global  existence of the solution.
\begin{proposition}[Well-posedness theory]\label{well-p}
Assume that  $3\leq d \leq 6$ and $a>-\frac{(d-2)^2}4+\left(\frac{(d-2)p-d}{2p}\right)^2$.

 \begin{enumerate}
   \item[\rm(i)] (Existence)
For any initial data $(u_0,u_1)\in \dot H^1\times L^2(\mathbb{R}^d)$,
there are a time interval $I\owns0$ and
a unique  solution $u$ to  \eqref{wave-La-p} on $I$, such that
$(u,u_t)\in C(I,\dot H^1\times L^2(\mathbb{R}^d)) $ and
 $ u \in  L_{t}^{\frac{2p}{(d-2)p-d} } L^{2p}_x(I\times \mathbb{R}^d). $

\item[\rm(ii)] (Long time perturbation)
Let $I=[0,T]$ be a   time interval with $0<T<\infty$. For any positive $M<\infty$, there exists $\varepsilon_0=\varepsilon_0(T,M)$ such that if $0<\varepsilon<\varepsilon_0$, then the following holds:
If $\tilde u$ be a solution to $\partial_{tt} \tilde u+\mathcal{L}_a \tilde u =-|\tilde u|^{p-1}\tilde u$
and
\begin{align*}
& \|\tilde u(0),\partial_t \tilde u(0)\|_{\dot H^1\times L^2(\Rm^d)}<\infty ,\\
&|\tilde u\|_{L_{t}^{\frac{2p}{(d-2)p-d} } L^{2p}_x(I\times \Rm^d)} <M, \\
&\|(u_0,u_1)-(\tilde u(0),\partial_t \tilde u(0))\|_{\dot H^1\times L^2(\Rm^d)}<\varepsilon.   \end{align*}
Then there exists a solution $u$ to \eqref{wave-La-p} on $I$ such that
\begin{align*}
&   \|u-\tilde u\|_{L_{t}^{\frac{2p}{(d-2)p-d} } L^{2p}_x (I\times \Rm^d)} < C(T,M)\varepsilon, \\
& \sup_{t\in I}\|(u(t),\partial u(t))-(\tilde u(t),\partial_t \tilde u(t))\|_{\dot H^1\times L^2(\Rm^d)}< C(T,M)\varepsilon.\end{align*}
\item[\rm(iii)](energy conservation)
If  $E(u_0,u_1)<\infty$, then we have
$E(u(t),u_t(t))\equiv E(u_0,u_1),$ for $t\in I_{max}$.
 \end{enumerate}
 By {\rm(i)} and {\rm (iii)},  we can obtain the global
well-posedness of \eqref{wave-La-p}.
\end{proposition}

\begin{proof}
 The proof of {\rm (i)} and  {\rm (ii)} follows from the estimates
\eqref{nonlinear-1}-\eqref{nonlinear-2} and standard fixed point argument, see for example
\cite{LS-1985, Shen-2014}.
Moreover,
by $(ii)$ and 
Fatou's Lemma,
 time reverse invariance of \eqref{wave-La-p}, it suffices to consider the
initial data in $(u_0,u_1)\in H^1\times L^2(\Rm^d)$.
From $u_0\in L^2(\Rm^d)$  and $(u,\partial_t u)\in C(I,\dot H\times L^2(\Rm^d)),$ we have
 $(u,\partial_t u)\in C(I,  H^1\times L^2(\Rm^d)).$
Then,
  we have the energy conservation law, by Theorem 4.1 in Strauss\cite{Strauss-1966},  if we take $V= H^1(\Rm^d),H=W=L^2(\Rm^d), A(t)=\mathcal{L}_a$ and $Z=L^2((0,T);L^2(\Rm^d))$.

\end{proof}

\section{Energy Flux}\label{energyflux}
\setcounter{section}{3}\setcounter{equation}{0}

In this section, we consider the energy flux of the solutions given by Proposition \ref{well-p}.
 Given any region $\Sigma \subset \Rm^d$, we use the following notation to represent local energy for convenience.
 \begin{equation}\label{energy-flux-1}
  E(t;\Sigma) \doteq \int_{\Sigma} \Big(\frac{1}{2}|\nabla u(x,t)|^2 +\frac{a}{2} \frac{|u(x,t)|^2}{|x|^2}+\frac{1}{2}| u_t(x,t)|^2 + \frac{1}{p+1}|u(x,t)|^{p+1}\Big) dx.
\end{equation}
 We have the following energy flux formula for $t_2>t_1>\eta$
 \begin{align}\nonumber
  &E(t_2; B(0,t_2-\eta)) - E(t_1; B(0,t_1-\eta)) \\
 =& \frac{1}{\sqrt{2}} \int_{|x|=t-\eta, t_1<t<t_2} \Big(\frac{1}{2}|u_r+u_t|^2 +\frac{a}{2} \frac{|u|^2}{|x|^2}+\frac{1}{2}|\slashed{\nabla} u|^2 + \frac{1}{p+1}|u|^{p+1}\Big) dS.
 \label{energy-flux-2}\end{align}
Since  the local energy $E(t;\Sigma)$ remains uniformly bounded,  the energy flux above is always bounded by a constant multiple of energy. 
Letting $t_1\rightarrow \eta$ and $t_2 \rightarrow +\infty$, we obtain the energy flux through a whole light cone:
\begin{equation}\label{energy-flux-3}
 \int_{|x|=t-\eta} \left(\frac{1}{2}|u_r+u_t|^2 +\frac{a}{2} \frac{|u|^2}{|x|^2}+\frac{1}{2}|\slashed{\nabla} u|^2 + \frac{1}{p+1}|u|^{p+1}\right) dS \lesssim E.
\end{equation}
The limit process $t_1\rightarrow \eta$ and $t_2 \rightarrow +\infty$ is valid at least in the radial case although it involves the integral of the negative function $\frac{a}{2}\frac{|u|^2}{|x|^2}$. In fact,
\begin{itemize}
 \item \; The limit $t_1\rightarrow \eta$ can be done for a.e. $\eta \in \Rm$ by making use of
  $$  \int_{T_1}^{T_2} \int_{\Rm^d}  \Big(\frac{1}{2}|u_r+u_t|^2 +\frac{|a|}{2} \frac{|u|^2}{|x|^2}+\frac{1}{2}|\slashed{\nabla} u|^2 + \frac{1}{p+1}|u|^{p+1}\Big) dx dt \lesssim (T_2-T_1) E < +\infty.
 $$
  \item\; In the radial case, we have $u(x,t) \lesssim |x|^{-\frac{2(d-1)}{p+3}}$. Thus the integral $\int_{|x|=t-\eta>1} \frac{|u|^2}{|x|^2} dS$ is always finite.
   This fact tell us that the limit $t_2\rightarrow \infty$ can be done.
\end{itemize}
Next we have a key observation
\begin{equation}\label{energy-flux-4}
 \int_{|x|=t-\eta} |u_r+u_t|^2 dS \geq \Big(\frac{ d-2 }{2} \Big)^2 \int_{|x|=t-\eta} \frac{|u|^2}{|x|^2} dS.
\end{equation}
In fact,  it suffices to  apply the sharp Hardy inequality on  $\tilde{u}(x) = u(x, |x|+\eta) \in \dot{H}^1 (\Rm^d)$.
\begin{proposition}[Energy flux] \label{energy flux}
 Let $u$ be a finite-energy radial solution to \eqref{wave-La-p}. Then we have the energy flux formula
 \begin{align}
  &E(t_2; B(0,t_2-\eta)) - E(t_1; B(0,t_1-\eta))\nonumber \\
 =& \frac{1}{\sqrt{2}} \int_{|x|=t-\eta, t_1<t<t_2} \Big(\frac{1}{2}|u_r+u_t|^2 +\frac{a}{2} \frac{|u|^2}{|x|^2}+ \frac{1}{p+1}|u|^{p+1}\Big) dS.
 \label{energy-flux-5}\end{align}
 In addition, the full energy
 flux through the light cone
 satisfies
 \begin{equation}\label{energy-flux-6}
  \int_{|x|=t-\eta} \Big(|u_r+u_t|^2 + \frac{|u|^2}{|x|^2}+ |u|^{p+1}\Big) dS \lesssim E.
\end{equation}
\end{proposition}

\section{Morawetz estimates}\label{Morawetz-sec}
\setcounter{section}{4}\setcounter{equation}{0}

In this section, 
 we shall prove a class of  local  Morawetz estimates for solutions in Proposition \ref{well-p}. Please note that these Morawetz estimates hold for non-radial solutions as well.
For this purpose, we give a Hardy inequality on a local domain.
\begin{lemma}\label{cutoff energy}
 Let $u \in \dot{H}^1(\Rm^d)$. Then for any $R>0$, we have
 \begin{equation}\label{morawetz-eq-1}
  \int_{|x|<R} \Big(|u_r (x)|^2 + a\tfrac{|u(x)|^2}{|x|^2}\Big) dx =\int_{|x|< R}\Big |u_r + \tfrac{\sigma}{|x|}u(x)\Big|^2 dx  -\sigma \int_{|x|=R}  \tfrac{|u(x)|^2}{|x|} dS(x),
 \end{equation}
where
$$\sigma   =\sigma_a  \doteq \tfrac{d-2}{2}-\sqrt{\tfrac{(d-2)^2}{4}+a}.$$
 Moreover,   we have
 \begin{equation}\label{morawetz-eq-2}
  \int_{|x|<R} \Big(|u_r (x)|^2 + \frac{|u(x)|^2}{|x|^2}\Big) dx \lesssim_{d,a}  \int_{|x|< R}\Big |u_r (x)+ \tfrac{\sigma}{|x|}u(x)\Big|^2 dx + \int_{|x|=R}  \frac{|u(x)|^2}{|x|} dS(x).
  \end{equation}
\end{lemma}
\begin{proof}
Identity \eqref{morawetz-eq-1} follows the fact $a=\sigma^2-(d-2)\sigma$ and integration by parts. In addition, we may choose $a_1 = a -\varepsilon>-\frac{d-2}{2}$ with a small constant $\varepsilon>0$, define the corresponding $\sigma_1 > \sigma$ and obtain
\begin{align}
 \int_{|x|<R} \Big(|\nabla u(x)|^2 + a_1 \tfrac{|u(x)|^2}{|x|^2}\Big) dx & =\int_{|x|< R}\Big |\nabla u(x)+ \tfrac{\sigma_1}{|x|}u(x)\Big|^2 dx  -\sigma_1 \int_{|x|=R}  \tfrac{|u(x)|^2}{|x|} dS(x)\nonumber\\
 & \geq -\sigma_1 \int_{|x|=R}  \tfrac{|u(x)|^2}{|x|} dS(x). \label{morawetz-eq-11}
\end{align}
 Taking the difference of \eqref{morawetz-eq-1} and \eqref{morawetz-eq-11}, we have
\begin{align*}
 \int_{|x|<R} \varepsilon \tfrac{|u(x)|^2}{|x|^2} dx \leq \int_{|x|< R}\Big |\nabla u(x)+ \tfrac{\sigma}{|x|}u(x)\Big|^2 dx  +(\sigma_1-\sigma) \int_{|x|=R}  \tfrac{|u(x)|^2}{|x|} dS(x).
\end{align*}
Finally we may combine this with \eqref{morawetz-eq-1} again to finish the proof of \eqref{morawetz-eq-2}.
\end{proof}


\begin{proposition}[Morawetz estimates]\label{morawetz-prop-0}
 Let $u$ be a finite-energy solution to \eqref{wave-La-p}. Then for any time $T_1<T_2$ and radius $R>0$ we have
 \begin{align}\nonumber
   &\frac1{2R} \int_{T_1}^{T_2} \int_{|x|< R}
  \Big(|\nabla u|^2 +a \frac{|u|^2}{|x|^2}+| u_t|^2  + \frac{(d-1)(p-1)-2}{p+1} |u|^{p+1} \Big) dxdt  \\
 & + \frac{d-1}{4R^2}\int_{T_1}^{T_2} \int_{|x|=R} |u|^2dS(x)dt  +   \int_{T_1}^{T_2}\int_{|x|> R} \Big(\frac{(d-1)(p-1)}{2(p+1)}  \frac{|u|^{p+1}}{|x|}+  (a+\lambda_d) \frac{|u|^2}{|x|^3}\Big) dxdt \nonumber\\
 & + \int_{|x|<R} \left[J_u^+(x,T_2)+J_u^-(x,T_1)\right] dx \leq 2E + \big(|a| -\lambda_d\big) \int_{|x|>R} \frac{|u(x,T_1)|^2+|u(x,T_2)|^2}{2|x|^2} dx.
 \label{Morawetz}
 \end{align}
 The notation $J_u^\pm (x,t)$ represents
 \[
  J_u^\pm  = \frac12\left[\frac{R^2-|x|^2}{R^2} \Big |u_r + \tfrac{\sigma}{|x|}u\Big|^2 + (\mu_d+a-2\sigma)\frac{|u|^2}{R^2} + \left| \frac{|x|}{R} u_r + \frac{(d-1)u}{2R}\pm u_t\right|^2 + |\slashed{\nabla} u|^2\right].
 \]
 Here $\mu_d=\frac{d^2-1}{4}$, $\lambda_d=\frac{(d-1)(d-3)}{4}$ and $\sigma= \frac{d-2}{2}-\sqrt{\left(\tfrac{d-2}{2} \right)^2+a}$ are all constants. Our assumption on $d$ and $a$ guarantees that $\mu_d+a-2\sigma \geq 3/4$.
\end{proposition}
\begin{proof}[Proof]\; We will calculate as though the solution is sufficiently smooth, otherwise we may apply smooth approximation techniques. Let
 \begin{align}
M(t)= \int_{|x|\leq R} u_t \Big( x\cdot \nabla u + \frac{d-1}{2} u\Big) dx + R\int_{|x|\geq R} u_t \Big(\frac{x}{|x|} \cdot\nabla u +  \frac{d-1}{2} \frac{u}{|x|}\Big) dx.
\end{align}
A straightforward calculation shows that the sum of double integrals in the left hand side of \eqref{Morawetz}, 
i.e.,
\begin{equation}
  \hbox{sum of first two lines of \eqref{Morawetz}} = -\frac{1}{R} \int_{T_1}^{T_2} M'(t) dt = \frac{1}{R}(M(T_1)-M(T_2)). \label{expression of left}
\end{equation}
Next we find upper bounds of $-M(T_2)/R$ and $M(T_1)/R$. By Cauchy-Schwarz and integration by parts, we have
\begin{align}
 -2M(T_2)/R& = -2\int_{|x|\leq R} u_t \Big( \frac{|x|}{R} u_r + \frac{d-1}{2} \frac{u}{R}\Big) dx -2 \int_{|x|\geq R} u_t \Big(u_r +  \frac{d-1}{2} \frac{u}{|x|}\Big) dx
 \nonumber\\
&    \leq   \int_{|x| \leq R} \left(\frac{|x|}{R} u_r + \frac{d-1}{2} \frac{u}{R}\right)^2 dx +   \int_{|x| > R} \left( u_r +  \frac{d-1}{2} \frac{u}{|x|}\right)^2 dx
 \nonumber\\
 & \qquad +  \int_{\Rm^d} |u_t|^2 dx -  \int_{|x|<R}  \left|u_t + \frac{|x|}{R} u_r + \frac{d-1}{2}\cdot \frac{u}{R}\right|^2 dx \nonumber \\
& = \frac{1}{R^2}\Big[\int_{|x|\leq R} \Big(  |x|^2 |u_r|^2  -\mu_d |u|^2\Big)dx + R^2 \int_{|x|\geq R}  \Big( |u_r|^2  -\lambda_d \frac{|u|^2}{|x|^2} \Big)dx\Big]\nonumber\\
 & \qquad +  \int_{\Rm^d} |u_t|^2 dx -  \int_{|x|<R}  \left|u_t + \frac{|x|}{R} u_r + \frac{d-1}{2}\cdot \frac{u}{R}\right|^2 dx.\label{upper Mt}
 \end{align}
We denote the sum enclosed in the middle bracket in \eqref{upper Mt} by $I$. Then we have
\begin{align*}
I =  & \int_{|x| \leq R} \Big( |x|^2 |u_r|^2  + a|u|^2\Big)dx -(\mu_d+a) \int_{|x|\leq R} |u|^2 dx\\
 & \qquad  + R^2 \int_{|x|\geq R} \Big(|\nabla u|^2 + a\frac{|u|^2}{|x|^2}\Big) dx -(\lambda_d+a)  R^2 \int_{|x|>R} \frac{|u|^2}{|x|^2} dx\\
 = & R^2 \int_{\Rm^d} \Big(|u_r|^2 + a\frac{|u|^2}{|x|^2}\Big) dx  - \int_0^R 2r \left[\int_{|x|<r} \Big(|u_r|^2 + a\frac{|u|^2}{|x|^2}\Big) dx\right] dr \\
 & \qquad -(\mu_d+a) \int_{|x|\leq R} |u|^2 dx -(\lambda_d+a) R^2 \int_{|x|>R} \frac{|u|^2}{|x|^2} dx.
\end{align*}
We obtain by Lemma \ref{cutoff energy}
\begin{align*}
 I & = R^2 \int_{\Rm^d} \Big(|u_r|^2 + a\frac{|u|^2}{|x|^2}\Big) dx - \int_{|x|<R} (R^2-|x|^2) \Big |u_r + \tfrac{\sigma}{|x|}u\Big|^2 dx\\
 &\qquad -(\mu_d+a-2\sigma) \int_{|x|\leq R} |u|^2 dx -(\lambda_d+a) R^2 \int_{|x|>R} \frac{|u|^2}{|x|^2} dx.
\end{align*}
Inserting this identity in \eqref{upper Mt} and recalling
\[
 E = \int_{\Rm^d} \left(\frac{1}{2} |u_r|^2 + \frac{1}{2}|\slashed{\nabla} u|^2 + \frac{a}{2} \cdot \frac{|u|^2}{|x|^2} + \frac{1}{2} |u_t|^2  + \frac{1}{p+1}|u|^{p+1}\right) dx,
\]
we have
\begin{align*}
 \frac{-M(T_2)}{R}  \leq E - \int_{|x|<R} J_u^+(x,T_2) dx  - \Big(\lambda_d+a\Big) \int_{|x|>R} \frac{|u(x,T_2)|^2}{2|x|^2} dx.
\end{align*}
Similarly we have
\begin{align*}
 \frac{M(T_1)}{R}  \leq E - \int_{|x|<R} J_u^-(x,T_1) dx  - \Big(\lambda_d+a\Big) \int_{|x|>R} \frac{|u(x,T_1)|^2}{2|x|^2} dx.
\end{align*}
Combining these upper bounds with \eqref{expression of left},
 we finish the proof of Proposition \ref{morawetz-prop-0}.
\end{proof}
\begin{corollary} \label{retard fast energy}
 Let $u$ be a finite-energy solution to \eqref{wave-La-p}. Then for any $R>0$ we have
 \begin{align}
 \int_{|x|<R} J_u^+(x,R) dx \lesssim \frac{1}{R}\int_{-R}^R \int_{|x|> R} \Big(|\nabla u|^2 + | u_t|^2 + |u|^{p+1}\Big) dxdt + \sum_{\pm} \int_{|x|>R} \frac{|u(x,\pm R)|^2}{|x|^2} dx. \label{morawetz-eq-6}
 \end{align}
\end{corollary}
\begin{proof}
 We apply Morawetz estimate \eqref{Morawetz} with $T_1 = -R$ and $T_2 = R$. Our assumption $p\geq 1+\frac{4}{d-1}$ means that $\frac{(d-1)(p-1)-2}{2(p+1)} \geq \frac{1}{p+1}$.
 Thus we may ignore all other positive terms in the left hand side of Morawetz estimate \eqref{Morawetz} and obtain
 \begin{align*}
  \frac{1}{R}\int_{ -R}^R \int_{|x|< R} & \Big(\frac{1}{2}|\nabla u|^2 +\frac{a}{2} \frac{|u|^2}{|x|^2}+\frac{1}{2}|u_t|^2 + \frac{1}{p+1}|u|^{p+1}\Big) dxdt   + \int_{|x|<R}J_u^+(x,R) dx\\
   & \leq 2E +(|a|-\lambda_d)\left[\int_{-R}^R \int_{|x|>R} \frac{|u|^2}{|x|^3} dxdt + \sum_{\pm} \int_{|x|>R} \frac{|u(x,\pm R)|^2}{2|x|^2} dx\right].
 \end{align*}
Thanks to the energy conservation law, we may substitute $2E$ in \eqref{Morawetz} by
 \begin{align*}
   \frac{1}{R} \int_{-R}^R \int_{\mathbb R^d}\Big(\frac{1}{2}|\nabla u|^2 +\frac{a}{2} \frac{|u|^2}{|x|^2}+\frac{1}{2}| u_t|^2 + \frac{1}{p+1}|u|^{p+1}\Big) dx dt,
 \end{align*}
 then subtract the first term in the left hand side from both sides, and finally obtain
 \begin{align*}
   \int_{|x|<R} J_u^+(x,R) dx &\leq \frac{1}{R}\int_{-R}^R \int_{|x|> R} \Big(\frac{1}{2}|\nabla u|^2 +\frac{a}{2} \frac{|u|^2}{|x|^2}+\frac{1}{2}| u_t|^2 + \frac{1}{p+1}|u|^{p+1}\Big) dxdt\nonumber\\
  & \quad  +(|a|-\lambda_d)\left[\int_{-R}^R \int_{|x|>R} \frac{|u|^2}{|x|^3} dxdt + \sum_{\pm} \int_{|x|>R} \frac{|u(x,\pm R)|^2}{2|x|^2} dx  \right]\\
  & \lesssim \frac{1}{R}\int_{-R}^R \int_{|x|> R} \Big(|\nabla u|^2 + | u_t|^2 + |u|^{p+1}\Big) dxdt + \sum_{\pm} \int_{|x|>R} \frac{|u(x,\pm R)|^2}{|x|^2} dx.
 \end{align*}
 Here we apply Hardy's inequality in the exterior region
 \[
  \int_{|x|>R} \frac{|u|^2}{|x|^3} dxdt \leq \frac{1}{R} \int_{|x|>R} \frac{|u|^2}{|x|^2} dx \lesssim \frac{1}{R} \int_{|x|>R} |\nabla u|^2 dx.
 \]
  This completes the proof of  Corollary \ref{retard fast energy}.
\end{proof}

\section{Method of Characteristic Lines}\label{mtheod-characteristic}
\setcounter{section}{5}\setcounter{equation}{0}

In this section, we prove Proposition  \ref{main 1} by method of characteristic lines.
  First we state the existence and uniqueness of a radiation
field for a finite energy solution to the linear wave equation.

\begin{proposition}[Radiation  field, \cite{DKM-2019,Friedlander-1962,Friedlander-1980,shenhdradial}]\label{radiation}
For any radial free wave solution $v$ with finite energy, there exists a
unique function $g_+\in L^2(\Rm)$ such that
\begin{equation}\label{characteristic-eq-1}
  \lim_{t\to \infty} \int_0^\infty \Big( \Big| r^{\frac{d-1}{2}} u_r(r,t)+g_+(t-r) \Big|^2 +  \Big| r^{\frac{d-1}{2}} u_t(r,t)-g_+(t-r) \Big|^2  \Big)dr=0.
\end{equation}
  On the other hand, for $g_+\in L^2(\Rm)$, there exists a
unique   radial free wave solution $u$ with finite energy 
 such that the equality \eqref{characteristic-eq-1}  
  holds.
\end{proposition}
Next, we recall the decay property of free waves.
\begin{proposition}[See Lemma 2.13 in \cite{shenhdradial}]\label{decay-out-cone}
Assume $d\geq3$. Let $u$ be a solution  to the free wave equation
$\big(\partial_t^2-\Delta\big) v=0$ with initial data $\dot H^1\times L^2(\Rm^d).$ Then we have
\begin{equation}
\lim_{\eta\to\infty} \sup_{t>\eta}
\Big\{ \int_{|x|<t-\eta}  \Big(|\nabla v(x,t)|^2+|v_t(x,t)|^2\Big)dx \Big\} =0,
\label{free limit 2}
\end{equation}
\vspace{-0.2cm}
\begin{equation}
  \lim_{R \rightarrow \infty} \sup_{t>0} \Big\{\int_{|x|>t+R} \Big(|\nabla v(x,t)|^2 + |v_t(x,t)|^2\Big) dx \Big\} = 0, \label{free limit 3}
\end{equation}
\end{proposition}

Assume that $u$ is a finite-energy, radial solution to \eqref{wave-La-p}. Let $w = r^{\frac{d-1}{2}}u(r,t)$,
then, we have
\begin{equation}\label{characteristic-eq-2}
(\partial_t^2 -\partial_r^2)w=- r^{\frac{d-1}{2}} |u|^{p-1} u-(\lambda_d+a) r^{\frac{d-5}{2} } u.
\end{equation}

As in \cite{shenenergy,shenhdradial}, by the characteristic method, radial Sobolev  and the integral estimate \eqref{energy-flux-6}, we have the following.
\begin{proposition} \label{variation of w}
 For $\tau+1 < t_1 < t_2 < s-1$, we have
\begin{align*}
|(w_t+w_r)(s-t_2,t_2)-(w_t+w_r)(s-t_1,t_1)| &\lesssim (s-t_2)^{-\beta(d,p)},\\
|(w_t-w_r)( t_2-\tau,t_2)-(w_t-w_r)(t_1-\tau,t_1)| &\lesssim (t_1-\tau)^{-\beta(d,p)},
\end{align*}
where $\beta(d,p)= \frac{(d-1)(p-1)-2}{2(p+1)}\in (0,\frac12)$  by the assumption of $p$.
\end{proposition}
\begin{proof}
 These two inequalities can be proved in the same way. Here we only stretch the proof of the first one. By \eqref{characteristic-eq-2} we have
\[
 \frac{\partial}{\partial t} (w_t+w_r)(s-t,t) = -(s-t)^{\frac{d-1}{2}} |u|^{p-1} u(s-t,t) - (\lambda_d+a) (s-t)^{\frac{d-5}{2}} u(s-t,t).
\]
Thus
\begin{align*}
 |(w_t+w_r)&(s-t_2,t_2)-(w_t+w_r)(s-t_1,t_1)|\\
  & \lesssim \int_{t_1}^{t_2} \left[(s-t)^{\frac{d-1}{2}} |u(s-t,t)|^{p} + (s-t)^{\frac{d-5}{2}} |u(s-t,t)|\right] dt.
\end{align*}
We may apply H\"{o}lder inequality and the integral estimate \eqref{energy-flux-6} to obtain
\begin{align*}
 \int_{t_1}^{t_2} & (s-t)^{\frac{d-1}{2}} |u(s-t,t)|^{p} dt\\
  & \lesssim \left(\int_{t_1}^{t_2} (s-t)^{-\frac{(d-1)(p-1)}{2}}dt\right)^{\frac{1}{p+1}} \left(\int_{t_1}^{t_2} (s-t)^{d-1} |u(s-t,t)|^{p+1} dt\right)^{\frac{p}{p+1}}\\
& \lesssim (s-t_2)^{-\beta(d,p)}.
\end{align*}
In addition, we may use the inequality $|u(r,t)| \lesssim r^{-(d-2)/2}$ and obtain
\[
 \int_{t_1}^{t_2} (s-t)^{\frac{d-5}{2}} |u(s-t,t)| dt \lesssim (s-t_2)^{-1/2}.
\]
Combining these two estimates, we complete the proof.
\end{proof}

As a  consequence of Proposition 5.3, we  immediately have the following.

\begin{proposition}
There exist $g_\pm\in L^2(\Rm)$ satisfying that
\begin{align*}
 &\lim_{t \rightarrow +\infty} (w_t-w_r)(t-\eta,t) = 2g_+(\eta),& &\|g_+\|_{L^2(\Rm)} \lesssim  E/c_d,&\\
 &\lim_{t \rightarrow -\infty} (w_t+w_r)(s-t,t) = 2g_-(s),& &\|g_-\|_{L^2(\Rm)} \lesssim   E/c_d.&
\end{align*}
\end{proposition}

\noindent
\textbf{Proof of Theorem \ref{main 1}}\; As a consequence of Proposition \ref{variation of w}, one can verify that
\begin{equation}\label{characteristic-eq-3}
 \limsup_{t\rightarrow +\infty} \int_{t-c\cdot t^{1-\kappa_0}}^{t+R} (|(w_t+w_r)(r,t)-2g_-(r+t)|^2 + |(w_t-w_r)(r,t)-2g_+(t-r)|^2) dr \lesssim c.
\end{equation}
where  $1-\kappa_0=2\beta(d,p)  $,  $c, R$ are positive constants.
A simple calculation shows
$$\int_{t-c\cdot t^{1-\kappa_0}}^{t+R} |g_-(r+t)|^2 dr \rightarrow 0,\quad\;\text{\rm as}\quad\; t\to\infty.$$
  Therefore, we have
$$  \limsup_{t\rightarrow +\infty} \int_{t-c\cdot t^{1-\kappa_0}}^{t+R} (|(w_t+w_r)(r,t)|^2 + |(w_t-w_r)(r,t)-2g_+(t-r)|^2) dr \lesssim c. $$
This can be rewritten as
 $$ \limsup_{t\rightarrow +\infty} \int_{t-c\cdot t^{1-\kappa_0}}^{t+R} (|w_t(r,t)-g_+(t-r)|^2 + |w_r(r,t)+g_+(t-r)|^2) dr \lesssim c. $$
Note that $w_r = r^{\frac{d-1}{2}} u_r + {\frac{d-1}{2}} r^{\frac{d-3}{2}}u$ and
$$ \lim_{t\rightarrow +\infty} \int_{t-c\cdot t^{1-\kappa_0}}^{t+R} r^{d-3 }|u(r,t)|^2 dr \lesssim \lim_{t\rightarrow +\infty}\int_{t-c\cdot t^{1-\kappa_0}}^{t+R}  r^{d-3 } r^{-\frac{4(d-1)}{p+3}}dr = 0,
$$
  by Lemma \ref{pointwise estimate 2}.
 Thus, we have
\begin{align}
  & \limsup_{t\rightarrow +\infty} \int_{t-c\cdot t^{1-\kappa_0}}^{t+R} (|r^{\frac{d-1}{2}} u_t(r,t)-g_+(t-r)|^2 + |r^{\frac{d-1}{2}} u_r(r,t)+g_+(t-r)|^2) dr \lesssim c .\label{u middle 2}
\end{align}
Similarly we have for $\eta, R>0$
\begin{equation} \label{u middle 1}
 \limsup_{t\rightarrow +\infty} \int_{t-\eta}^{t+R} (|r^{\frac{d-1}{2}} u_t(r,t)-g_+(t-r)|^2 + |r^{ \frac{d-1}{2}} u_r(r,t)+g_+(t-r)|^2) dr = 0.
\end{equation}

Next, by Proposition \ref{radiation}, there exists free wave $v$ such that
%
\begin{align}
 & \lim_{t\rightarrow +\infty} \int_{0}^{\infty} (|r^{ \frac{d-1}{2}} v_t(r,t)-g_+(t-r)|^2 + |r^{ \frac{d-1}{2}} v_r(r,t)+g_+(t-r)|^2) dr = 0. \label{free limit 1}
\end{align}
Combining \eqref{free limit 1} with \eqref{u middle 1} or \eqref{u middle 2}, we have
\begin{align}
 & \lim_{t \rightarrow +\infty} \int_{t-\eta<|x|<t+R} \Big(|\nabla u(x,t) - \nabla v(x,t)|^2 + |u_t(x,t)-v_t(x,t)|^2\Big) dx = 0, \label{difference middle 1}\\
 & \limsup_{t \rightarrow +\infty} \int_{t-c\cdot t^{1-\kappa_0}<|x|<t+R} \Big(|\nabla u(x,t) - \nabla v(x,t)|^2 + |u_t(x,t)-v_t(x,t)|^2\Big) dx \lesssim c. \label{difference middle 2}
\end{align}
Finally, 
to complete the proof of Theorem \ref{main 1}, it suffices  to show
\begin{equation}
 \lim_{R \rightarrow +\infty} \sup_{t>0} \Big\{\int_{|x|>t+R} (|\nabla u(x,t)-\nabla v(x,t)|^2 + |u_t(x,t)-v_t(x,t)|^2) dx \Big\} = 0. \label{difference outside}
\end{equation}
In fact, this is a direct consequence of
the finite propagation speed.
One can also see  \eqref{free limit 3}, or employ   Lemma \ref{decay of tail} with $\kappa=0$ if the decay of energy is unknown.
%
%
%

\section{Further estimates with energy decay and the proof of Theorem \ref{main 2}  }\label{further}
\setcounter{section}{6}\setcounter{equation}{0}

In this section, we consider the scattering theory of
\eqref{wave-La-p} under the condition of
 additional energy decay \eqref{weight-energy}.

\begin{lemma} \label{decay of tail}
Let $\kappa \in [0,1)$. Assume that initial data $(u_0,u_1)\in \dot{H}^1 \times L^2$ satisfy
 \[
  E_\kappa (u_0,u_1) = \int_{\Rm^d} (|x|^\kappa + 1)\left(|\nabla u_0(x)|^2 + |u_1(x)|^2 + |u_0(x)|^{p+1}\right) dx < +\infty.
 \]
 Then the corresponding solution $u$ to \eqref{wave-La-p} satisfies for each $r>0$
 \begin{align}\label{basicfurther-eq-0}
  \int_{|x|>r+|t|} \Big(|\nabla u(x,t)|^2 + \frac{|u(x,t)|^2}{|x|^2} + |u_t(x,t)|^2 + |u(x,t)|^{p+1}\Big)dx \lesssim r^{-\kappa} E_{\kappa,r}(u_0,u_1),
 \end{align}
 where  $E_{\kappa,r}(u_0,u_1)$ is defined by
 \[
  E_{\kappa,r}(u_0,u_1) \doteq \int_{|x|>r/2} (|x|^\kappa + 1)\left(|\nabla u_0(x)|^2 + |u_1(x)|^2 + |u_0(x)|^{p+1}\right) dx \rightarrow 0, \; \hbox{as}\; r \rightarrow +\infty.
 \]
 If the initial data is also radial, we have
 \begin{align}
  |u(x,t)|\lesssim |x|^{-\frac{2(d-1)}{p+3}}(|x|-|t|)^{-\frac{2\kappa}{p+3}} E_{\kappa,|x|-|t|}(u_0,u_1)^{\frac{2}{p+3}},\;\;|x|>|t|.\label{basicfurther-eq-1}
 \end{align}
\end{lemma}
\begin{proof}
 Let us fix a radial, smooth cut-off function $\phi: \Rm^d \rightarrow [0,1]$ satisfying
\[
 \phi(x) = \left\{
 \begin{array}{ll} 0, & |x|\leq 1/2; \\
 1, & |x|\geq 1,\end{array}\right.
\]
and define $(u_{0,r}, u_{1,r}) = \phi(x/r)(u_0,u_1)$. A simple calculation shows
\begin{align*}
E(u_{0,r}, u_{1,r}) & \lesssim \int_{\Rm^d} (|\nabla u_{0,r}|^2 + |u_{1,r}|^2+|u_{0,r}|^{p+1}) dx \\
 & \lesssim \int_{|x|>r/2} (|\nabla u_0|^2 + |u_1|^2+|u_0|^{p+1}) dx + \int_{r/2<|x|<r} \frac{|u_0|^2}{|x|^2} dx\\
 & \lesssim \int_{|x|>r/2} (|\nabla u_0|^2 + |u_1|^2+|u_0|^{p+1}) dx\\
 & \lesssim r^{-\kappa}  E_{\kappa,r}(u_0,u_1).
\end{align*}
Let $u^{(r)}$ be the solution to \eqref{wave-La-p} with initial data $(u_{0,r}, u_{1,r})$. By energy conservation law we immediately have
\begin{align}
 \int_{\Rm^d} \Big(|\nabla u^{(r)}|^2 + \frac{|u^{(r)}|^2}{|x|^2} + |u_t^{(r)}|^2+|u^{(r)}|^{p+1}\Big) dx & \lesssim E(u_{0,r}, u_{1,r}) \lesssim r^{-\kappa} E_{\kappa,r}(u_0,u_1). \label{energy estimate tail}
\end{align}
This proves \eqref{basicfurther-eq-0} since we always have $u(x,t) \equiv u^{(r)}(x,t)$ for $|x|>r+|t|$ by finite speed of propagation. If the data is also radial, then we may apply Lemma \ref{pointwise estimate 2} to conclude
\begin{align*}
 |u(x,t)| & = |u^{(|x|-|t|)}(x,t)| \lesssim |x|^{-\frac{2(d-1)}{p+3}} \|u^{(|x|-|t|)}(\cdot,t)\|_{\dot{H}^1}^{\frac{2}{p+3}} \|u^{(|x|-|t|)}(\cdot,t)\|_{L^{p+1}}^{\frac{p+1}{p+3}}\\
 &  \lesssim |x|^{-\frac{2(d-1)}{p+3}} \left((|x|-|t|)^{-\kappa} E_{\kappa,|x|-|t|}(u_0,u_1)\right)^{\frac{2}{p+3}},
\end{align*}
where  we have used \eqref{energy estimate tail} with $r = |x|-|t|$.
\end{proof}
\begin{corollary} \label{integral estimate kappa}
 Let $(u_0,u_1)$ be as in Theorem \ref{main 2}. The corresponding solution $u$ satisfies
 \begin{align*}
  \int_{|x|>R} \frac{|u(x,t)|^2}{|x|^2} dx  & \lesssim R^{-\frac{p+5}{p+3}\kappa_0} E_{\kappa_0} (u_0,u_1)^{\frac{4}{p+3}},\;\;|t|\leq R.
 \end{align*}
\end{corollary}
\begin{proof}
 By the pointwise estimate given in Lemma \ref{decay of tail}, if $|t|\leq R$ we have
  \begin{align*}
  \int_{|x|>R} \frac{|u(x,t)|^2}{|x|^2} dx & \lesssim \int_{|x|>R} |x|^{-\frac{4(d-1)}{p+3}-2}(|x|-|t|)^{-\frac{4\kappa_0}{p+3}} E_{\kappa_0,|x|-|t|}(u_0,u_1)^{\frac{4}{p+3}} dx\\
  & \lesssim R^{-\frac{p+5}{p+3}\kappa_0} E_{\kappa_0} (u_0,u_1)^{\frac{4}{p+3}}.
 \end{align*}
\end{proof}

\begin{proposition} \label{energy estimate inside}
 Let $(u_0,u_1)$ be as in Theorem \ref{main 2}, and let $c > 0$. The corresponding solution $u$ satisfies
 \begin{equation}
  \lim_{t\rightarrow +\infty} \int_{|x|<t-c\cdot t^{1-\kappa_0}} \Big(|\nabla u|^2 + \frac{|u|^2}{|x|^2} + |u_t|^2 \Big) dx = 0. \label{basicfurther-eq-3}
 \end{equation}
\end{proposition}
\begin{proof}
We start by recalling the conclusion of Corollary \ref{retard fast energy}
 \begin{align*}
 \int_{|x|<t} J_u^+(x,t) dx &\lesssim \frac{1}{t}\int_{-t}^t \int_{|x|> t} \Big(|\nabla u|^2 + |u_t|^2 + |u|^{p+1}\Big) dxdt' + \sum_{\pm} \int_{|x|>t} \frac{|u(x,\pm t)|^2}{2|x|^2} dx\\
 & \doteq I_1+I_2.
 \end{align*}
 Upper bound of $I_2$ can be found in Corollary \ref{integral estimate kappa}:
\[
 I_2 \lesssim t^{-\frac{p+5}{p+3}\kappa_0} E_{\kappa_0}(u_0,u_1).
\]
 We may also find an upper bound of $I_1$ by applying \eqref{basic-eq-2} and Lemma \ref{decay of tail}:
\begin{align*}
I_1 & \lesssim \frac{1}{t}\left[t^{(1-\kappa_0)/2} E +  \int_{-t+t^{(1-\kappa_0)/2}}^{t-t^{(1-\kappa_0)/2}} \int_{|x|> t} \Big(|\nabla u|^2 +| u_t|^2 + |u|^{p+1}\Big) dxdt'\right]\\
 & \lesssim \frac{1}{t}\left[t^{(1-\kappa_0)/2} E +  \int_{-t+t^{(1-\kappa_0)/2}}^{t-t^{(1-\kappa_0)/2}} (t-|t'|)^{-\kappa_0} E_{\kappa_0, t^{(1-\kappa_0)/2}} (u_0,u_1) dt'\right]\\
 & \lesssim t^{(1-\kappa_0)/2-1} E + t^{-\kappa_0} E_{\kappa_0, t^{(1-\kappa_0)/2}} (u_0,u_1).
\end{align*}
Combining these upper bounds we have
\[
  \int_{|x|<t} J_u^+(x,t) dx = o(t^{-\kappa_0}).
\]
We recall the definition of $J_u^+$ given in Proposition \ref{morawetz-prop-0} and obtain
\begin{align}
 \int_{|x|<t} \left( \frac{|u|^2}{t^2} + \left| \frac{|x|}{t} u_r + \frac{(d-1)u}{2t}+ u_t\right|^2 + |\slashed{\nabla} u|^2\right) dx & = o(t^{-\kappa_0}).\label{J1 estimate}\\
 \int_{|x|<t} \frac{t^2-|x|^2}{t^2} \Big |u_r + \tfrac{\sigma}{|x|}u\Big|^2 dx & = o(t^{-\kappa_0}). \label{J2 estimate}
\end{align}
If $|x|<t-c\cdot t^{1-\kappa_0}$, then we have $\frac{t^2-|x|^2}{t^2} \geq c t^{-\kappa_0}$. Thus \eqref{J2 estimate} implies
\[
 \lim_{t\rightarrow +\infty} \int_{|x|<t-c\cdot t^{1-\kappa_0}} \Big |u_r + \tfrac{\sigma}{|x|}u\Big|^2 dx = 0.
\]
In addition, the universal upper bound $  |u(x,t)| \lesssim |x|^{-\frac{2(d-1)}{p+3}}$ in Lemma \ref{pointwise estimate 2} gives
\[
 \lim_{t\rightarrow +\infty} \int_{|x|=t-c\cdot t^{1-\kappa_0}} \frac{|u(x,t)|^2}{|x|} dS(x) = 0.
\]
Thus we may apply Lemma \ref{cutoff energy} and obtain
\begin{equation} \label{J3 estimate}
 \lim_{t\rightarrow +\infty} \int_{|x|<t-c\cdot t^{1-\kappa_0}} \left(|u_r|^2 + \frac{|u|^2}{|x|^2}\right) dx = 0.
\end{equation}
Finally we observe the inequality
\[
 |u_t|^2 \lesssim \frac{|u|^2}{t^2} + \left| \frac{|x|}{t} u_r + \frac{(d-1)u}{2t}\pm u_t\right|^2 + |u_r|^2
\]
and complete the proof by combining \eqref{J1 estimate} and \eqref{J3 estimate}.
\end{proof}

Now, {\bf we are ready to prove Theorem \ref{main 2}}.
Let $u$ be a radial solution to \eqref{wave-La-p} with initial data $(u_0,u_1)$ satisfying $ E_{\kappa_0}(u_0,u_1) \lesssim E_{\kappa} (u_0,u_1)< +\infty$. Then, for $v$ given as in Section \ref{mtheod-characteristic}, we have
\begin{align*}
 \limsup_{t \rightarrow +\infty} \int_{\Rm^d} \Big(|\nabla u(x,t) - \nabla v(x,t)|^2 + |u_t(x,t)-v_t(x,t)|^2\Big) dx \leq L_1 + L_2 + L_3,
\end{align*}
with
\begin{align*}
 L_1 = & \limsup_{t \rightarrow +\infty} \int_{|x|<t-c\cdot t^{1-\kappa_0}} \Big(|\nabla u(x,t) - \nabla v(x,t)|^2 + |u_t(x,t)-v_t(x,t)|^2\Big) dx; \\
 L_2 = & \limsup_{t \rightarrow +\infty} \int_{t-c\cdot t^{1-\kappa_0}<|x|<t+R} \Big(|\nabla u(x,t) - \nabla v(x,t)|^2 + |u_t(x,t)-v_t(x,t)|^2\Big) dx;\\
 L_3 = & \limsup_{t \rightarrow +\infty} \int_{|x|>t+R} \Big(|\nabla u(x,t) - \nabla v(x,t)|^2 + |u_t(x,t)-v_t(x,t)|^2\Big) dx,
\end{align*}
where $c,R>0$ are constants to be determined.
We have already known $L_3 = 0$ and $L_2 \lesssim c$ by the conclusion of Theorem \ref{main 1} and \eqref{difference middle 2}. We also have $L_1 = 0$ by a combination of Proposition \ref{energy estimate inside} and \eqref{free limit 2}. In summary we have
\[
  \limsup_{t \rightarrow +\infty} \int_{\Rm^d} \Big(|\nabla u(x,t) - \nabla v(x,t)|^2 + |u_t(x,t)-v_t(x,t)|^2\Big) dx \lesssim c
\]
for all positive
constants
 $c$. We then make $c\rightarrow 0$ and finish the proof.

\appendix

\section{}

\setcounter{equation}{0}
\renewcommand\theequation{A.\arabic{equation}}

\subsection{The asymptotic behaviour of linear solution}
 Let $u$ be a finite-energy radial solution to the linear wave equation with an inverse-square potential
\[
 \left\{\begin{array}{l}  u_{tt} +\Big(- \Delta + \frac{a}{|x|^2} \Big) u = 0, \qquad (x,t)\in \Rm^d \times \Rm,\\
 (u,u_t)|_{t=0} = (u_0,u_1) \in \dot{H}^1 \times L^2. \end{array}\right.
\]
Here $d \geq 3$ and $a > -\tfrac{(d-2)^2}{4}$. We prove that there exists a free wave $v$, i.e., a solution to the linear wave equation $\partial_t^2 v - \Delta v = 0$, so that
\begin{equation} \label{appendix-eq-01}
 \lim_{t \rightarrow +\infty} \|(u(\cdot,t), u_t(\cdot,t)) - (v(\cdot,t), v_t(\cdot,t))\|_{\dot{H}^1 \times L^2 } = 0.
\end{equation}

The proof of \eqref{appendix-eq-01} divides into three steps. We first prove a few estimates in step 1.
Following  the same argument as in Lemma \ref{decay of tail}, one can obtain
\begin{align}
 &\int_{|x|>|t|+r} \Big(|\nabla u(x,t)|^2 + \frac{|u(x,t)|^2}{|x|^2} + |u_t(x,t)|^2\Big) dx  \lesssim E_{0,r}, \label{tail estimate 1 linear}\\
 &|u(x,t)|  \lesssim E_{0,r}^{1/2} |x|^{-\frac{d-2}{2}}, \quad |x|\geq r+|t|\label{tail estimate 1 linear-0}
\end{align}
such that
\begin{equation}\label{appendix-eq-1}
 E_{0,r} = \int_{|x|>r/2} \Big(|\nabla u_0(x)|^2 + |u_1(x)|^2\Big) dx \rightarrow 0, \quad \hbox{as}\; r\rightarrow \infty.
\end{equation}
We may combine \eqref{tail estimate 1 linear} with the universal estimate $|u(x,t)| \lesssim |x|^{-\frac{d-2}{2}}$ and obtain
\begin{equation} \label{improved H3}
 \lim_{t\rightarrow +\infty} \int_{|x|>t} \frac{|u(x,\pm t)|^2}{|x|^2} dx = 0.
\end{equation}
We also have by a direct computation
\begin{align*}
 |u(r_1,t) - u(r_2,t)| &= \int_{r_1}^{r_2} |u_r (r,t)| dr \lesssim \Big(\int_{r_1}^{r_2} r^{d-1} |u_r(r,t)|^2 dr\Big)^{1/2} \Big(\int_{r_1}^{r_2} r^{-(d-1)} dr\Big)^{1/2} \\
 & \lesssim (r_2-r_1)^{1/2} r_1^{-\frac{d-1}{2}}.
\end{align*}
Thus if $(1-\delta)|t|\leq r\leq |t|+r_0\leq (1+\delta)|t|$ with a small constant $\delta \ll 1$, then we have
\begin{equation} \label{improved pointwise H1}
 |u(r,t)| \leq |u(|t|+r_0,t)| + |u(r,t)- u(|t|+r_0,t)| \lesssim (E_{0,r_0}^{1/2} + \delta^{1/2}) r^{-\frac{d-2}{2}}.
\end{equation}
It immediately follows that
\begin{equation} \label{improved H2}
 \limsup_{t\rightarrow +\infty} \int_{|x|= (1-\delta) t} \frac{|u(x,t)|^2}{|x|} dS(x) \lesssim \delta.
\end{equation}
In step 2 we prove that there exists a finite-energy free wave $v$, so that for $0<\delta<1$,
\begin{equation} \label{outer part linear}
 \limsup_{t \rightarrow +\infty} \int_{|x|>(1-\delta)t} \Big(|\nabla u(x,t) - \nabla v(x,t)|^2 + |u_t(x,t)-v_t(x,t)|^2\Big) dx \lesssim \delta.
\end{equation}

Define $w = r^{\frac{d-1}{2}} u(r,t)$, then $w$ satisfies
\begin{equation}\label{appendix-eq-02}
  (\partial_t^2 -\partial_r^2)w = -(\lambda_d+a) r^{\frac{d-5}{2}} u.
\end{equation}
Note that $|u(r,t)| \lesssim r^{-\frac{d-2}{2}}$, applying  method of characteristic lines, we can obtain $g_+, g_- \in L^2 (\Rm)$ such that
\begin{align*}
&|(w_t+w_r)(r,t)-2g_-(r+t)| \lesssim r^{-1/2},& &r>1;&\\
&|(w_t-w_r)(r,t)-2g_+(t-r)| \lesssim r^{-1/2},& &r>1.&
\end{align*}
Thus we have ($\delta \ll 1$, $R>0$ are constants)
\[
 \limsup_{t\rightarrow +\infty} \int_{(1-\delta)t}^{t+R} (|w_t(r,t)-g_+(t-r)|^2 + |w_r(r,t)+g_+(t-r)|^2) dr \lesssim \delta.
\]
By the identity $w_r = r^{\frac{d-1}{2}} u_r + \frac{d-1}{2} r^{\frac{d-3}{2}} u$ and the upper bound
\[
 \limsup_{t\rightarrow +\infty} \int_{(1-\delta)t}^{t+R} r^{d-3} |u(r,t)|^2 dr \lesssim  \limsup_{t\rightarrow +\infty} \int_{(1-\delta)t}^{t+R} \frac{1}{r} dr \lesssim \delta.
\]
We have
\[
 \limsup_{t\rightarrow +\infty} \int_{(1-\delta)t}^{t+R} (|r^{\frac{d-1}{2}} u_t(r,t)-g_+(t-r)|^2 + |r^{\frac{d-1}{2}} u_r(r,t)+g_+(t-r)|^2) dr \lesssim \delta.
\]
By radiation fields, there exists a free wave $v$ so that
\[
 \limsup_{t \rightarrow +\infty} \int_{(1-\delta)t<|x|<t+R} \Big(|\nabla u(x,t) - \nabla v(x,t)|^2 + |u_t(x,t)-v_t(x,t)|^2\Big) dx \lesssim \delta.
\]
Combining this with \eqref{tail estimate 1 linear}, we  obtain \eqref{outer part linear}.

In step 3 we show
\begin{equation} \label{inner part linear}
 \limsup_{t \rightarrow +\infty} \int_{|x|<(1-\delta)t} \Big(|\nabla u(x,t) - \nabla v(x,t)|^2 + |u_t(x,t)-v_t(x,t)|^2\Big) dx \lesssim \delta.
\end{equation}
We follow a similar argument to Proposition \ref{morawetz-prop-0} and Corollary \ref{retard fast energy} (simply ignore the terms involving $|u|^{p+1}$), and obtain a Morawetz-type estimate
\begin{align*}
 \int_{|x|<t} J_u^+(x,t) dx
 & \lesssim \frac{1}{t}\int_{-t}^t \int_{|x|> t} \Big(|\nabla u|^2 + |u_t|^2 \Big) dxdt' + \sum_{\pm} \int_{|x|>t} \frac{|u(x,\pm t)|^2}{2|x|^2} dx.
 \end{align*}
We claim that the right hand side converges to zero as $t\rightarrow +\infty$. In fact we may show the convergence of the first term in the same way as in Proposition \ref{energy estimate inside} by \eqref{tail estimate 1 linear}. The convergence of the second term has been known \eqref{improved H3}. As a result
\[
 \lim_{t\rightarrow +\infty} \int_{|x|<t} J_u^+(x,t) dx = 0.
\]
We recall the definition of $J_u^+$ and obtain
\begin{align*}
 \lim_{t\rightarrow +\infty} \int_{|x|<t} \left( \frac{|u|^2}{t^2} + \left| \frac{|x|}{t} u_r + \frac{(d-1)u}{2t}\pm u_t\right|^2 + |\slashed{\nabla} u|^2\right) dx & = 0. \\ \lim_{t\rightarrow +\infty}  \int_{|x|<(1-\delta)t}  \Big |u_r + \tfrac{\sigma}{|x|}u\Big|^2 dx & = 0.
\end{align*}
Next we combine these estimates with \eqref{improved H2}, Proposition \ref{decay-out-cone}, and Lemma \ref{cutoff energy} to complete the proof of \eqref{inner part linear}, as we did in the proof of Proposition \ref{energy estimate inside}.  Finally we combine \eqref{outer part linear} and \eqref{inner part linear} to conclude
\[
 \limsup_{t \rightarrow +\infty} \int_{\Rm^d} \Big(|\nabla u(x,t) - \nabla v(x,t)|^2 + |u_t(x,t)-v_t(x,t)|^2\Big) dx \lesssim \delta.
\]
We then finish the proof by letting $\delta \rightarrow 0^+$.

\subsection{Square function estimates}

In this subsection, employing the Stein-Tomas restriction theorem in Lorentz space, we show
\begin{equation}\label{square-ets}
\Big\|\int_{\Rm} e^{is|\nabla|}  F(\cdot,s) ds \Big\|_{\dot H^{-\frac12(\Rm^d)}}
\leq C \|F\|_{L^{\tfrac{2d}{d+2},2 }_x L^2_t(\Rm\times\Rm^d) },
\end{equation}
for radial function $F\in {L^{\tfrac{2d}{d+2},2 }_x L^2_t(\Rm\times\Rm^d) }$.

 \begin{proof}

The proof follows the arguments of \cite{RV-1993}.
By duality, it suffices  to show
\begin{equation}
\| |\nabla |^{-\gamma}e^{it|\nabla|}  f \|_{L^{q,2}_x L^2_t(\Rm\times\Rm^d)} \leq C \|f\|_{L^2(\Rm^d)},
\end{equation}
where  $\gamma=\frac{d}2-\frac{1}{2}-\frac{d}{q}$ and $q\geq \frac{2(d+1)}{d-1}$.
Utilizing the polar coordinates, we have
\[
|\nabla |^{-\gamma}e^{it|\nabla|}  f(x) =\int_0^\infty e^{it r} r^{-\gamma} \int_{|\xi|=r} e^{ix \xi} \hat f(\xi) d\sigma^{n-1}_r(\xi) dr.
\]
By Plancherel theorem with respective to variable $t$, we have
\begin{align*}
\| |\nabla |^{-\gamma}e^{it|\nabla|}  f(x) \|_{L^{q,2}_xL^2_t} = &
 \Big\| \Big(  \int_0^\infty \Big|    \int_{|\xi|=r} e^{ix \xi} \hat f(\xi) d\sigma^{n-1}_r(\xi)  \Big|^2  r^{-2\gamma} dr \Big)^\frac12\Big\|_{L^{q,2}_x}\\
  \leq  &
 \Big(  \int_0^\infty    \Big\|    \int_{|\xi|=r} e^{ix \xi} \hat f(\xi) d\sigma^{n-1}_r(\xi)   \Big\|_{L^{q,2}_x}^2  r^{-2\gamma} dr \Big)^\frac12  \\
   \leq & \Big(  \int_0^\infty  \int_{|\xi|=r} | \hat f(\xi) |^2 d\sigma^{n-1}_r(\xi)   dr \Big)^\frac12
  \\
 = &  \|f\|_{L^2_x(\Rm^d)}.
\end{align*}
where we used the Stein-Tomas restriction theorem of the Lorentz spaces in \cite{Bak-Segger-2011} and the following Minkowski inequality.
\begin{align}\label{add-1}
\|f(x,y)\|_{L^{p,r}_x L^r_y} \le \|f(x,y)\|_{ L^r_y L^{p,r}_x}, \;\; 1\leq  r< p <\infty. 
 \end{align}
This completes the proof the square function estimate in Lorentz space.

For the sake of completeness, we give a proof of the  Minkowski inequalities in Lorentz space. In fact, for $1\leq  r< p<\infty$, we have
\begin{align*}
\|f(x,y)\|_{L^{p,r}_x L^r_y} = &
 \Big\| |f(x,y)|^r \Big\|^{\frac1r}_{L^{\frac{p}{r},1}_x L^1_y} \\
 \approx  & \Big( \sup_{\|g\|_{L^{(\frac{p}{r})',\infty}} \leq 1 }   \int |g(x)|\int |f(x,y)|^r dy   dx \Big)^\frac1r  \\
= & \Big( \sup_{\|g\|_{L^{(\frac{p}{r})',\infty}} \leq 1 }   \int \int |f(x,y)|^r  |g(x)| dx dy    \Big)^\frac1r \\
\leq &  \| |f|^r\|_{L^1_y L^{\frac{p}{r},1}_x}^\frac1r = \|f(x,y)\|_{ L^r_y L^{p,r}_x}.
 \label{add-2}\end{align*}
\end{proof}

\textbf{Acknowledgments:} We thank the anonymous referee and the associated editor for their invaluable comments
which helped to improve the paper. This project was suppported by the National Key R\&D program of China: No.2022YFA1005700.
C. Miao is supported in part by the NSF of China under grant No.11831004. R. Shen is supported by the NSF of China under grant No.11771325, No. 12071339.
T. Zhao is supported in part by the NSF  of China under grant No. 12101040 and the Fundamental Research  Funds for the Central Universities (FRF-TP-20-076A1).

\end{document}